\documentclass[graybox]{svmult}

\usepackage{mathptmx}       
\usepackage{helvet}         
\usepackage{type1cm}        
\usepackage{makeidx}         
\usepackage{graphicx}        
\usepackage{multicol}        
\usepackage[bottom]{footmisc}

\makeindex             

\usepackage{url}
\usepackage{amsmath}
\usepackage{amsfonts}
\usepackage{amssymb}

\DeclareMathOperator{\Expectation}{E} 
\newcommand{\absoluteval}[1]{\left\vert#1\right\vert}
\newcommand{\derivby}[1]{\frac{d}{d#1}}
\newcommand{\escortof}[1]{\operatorname{escort}\left(#1\right)}
\newcommand{\expectat}[2]{{\Expectation}_{#1}\left[#2\right]}

\newcommand{\normat}[2]{\left\Vert#2\right\Vert_{#1}}
\newcommand{\partiald}[2]{\frac{\partial}{\partial #1} #2}
\newcommand{\reals}{\mathbb{R}}

\newcommand{\set}[1]{\left\{#1\right\}}

\begin{document}

\title*{A class of non-parametric deformed exponential statistical models}
\titlerunning{Deformed exponential models}
\author{Montrucchio, Luigi and Pistone, Giovanni}
\authorrunning{L. Montrucchio and G. Pistone}
\institute{Luigi Montrucchio \at Collegio Carlo Alberto, Piazza Vincenzo Arbarello 8, 10122 Torino, Italy, \email{luigi.montrucchio@unito.it}
\and Giovanni Pistone \at de Castro Statistics, Collegio Carlo Alberto, Piazza Vincenzo Arbarello 8, 10122 Torino, Italy, \email{giovanni.pistone@carloalberto.org}}
\maketitle
\smartqed

\abstract{We study the class on non-parametric deformed statistical models where the deformed exponential has linear growth at infinity and is sub-exponential at zero. This class generalizes the class introduced by N.J.~Newton. We discuss the convexity and regularity of the normalization operator, the form of the deformed statistical divergences and their convex duality, the properties of the escort densities, and the affine manifold structure of the statistical bundle.}

\section{Introduction}\label{sec:introduction}
In this paper we study a geometry on the set $\mathcal P$ of strictly positive probability densities on a probability space $(\mathbb{X},\mathcal{X},\mu)$. In some cases one is led to consider the set $\overline{\mathcal P}$ of probability densities i.e., without the restriction of strict positivity. There is a considerable literature on the Information Geometry in the sense defined in the Amari and Nagaoka monograph \cite{amari|nagaoka:2000} on $\mathcal P$. There is also a non-parametric approach i.e., we are not considering the geometry induced on the parameter set of a given statistical model but on the full set of densities. This was done in \cite{pistone|sempi:95,pistone:2013GSI} by using logarithmic chart to represent densities. 

A different approach, that leads to the construction of an Hilbert manifold on $\mathcal P$, has been proposed by N.J.~Newton in \cite{newton:2012,newton:2016}. It is based on the use of the chart $p \mapsto p - 1 - \log p$ instead of a purely logarithmic chart. This paper presents a variation on the same theme by enlarging the class of permitted charts.

Let $\mathcal{M} \subset \mathcal P$. At each $p \in \mathcal M $, the Hilbert space of square-integrable random variables $L^2(p)$ provides a fiber that sits at $p \in \mathcal{M}$, so we can define the \emph{Hilbert bundle} with base $\mathcal{M}$. The Hilbert bundle, or similar bundles with fibers which are vector spaces of random variables, provides a convenient framework for Information Geometry, cf. \cite{amari:87dual,kass|vos:1997,pistone:2013GSI}.

If $\mathcal{M}$ is an exponential manifold in the sense of \cite{pistone|sempi:95}, there exists a splitting of
each fiber $L^2(p) = \mathcal {H}_p \oplus \mathcal {H}_p^\perp$, such that each $\mathcal {H}_p$ contains a dense vector sub-space which is an expression of the tangent space $T_p\mathcal M$ of the manifold. Moreover, the manifold on $\mathcal{M}$ is an affine manifold (it can be defined by an atlas whose transition mapping are affine) and it is also an Hessian manifold (the inner product on each fiber is the second derivative of a potential function, \cite{shima:2007}). 

When the sample space is finite and $\mathcal M$ is the full set $\mathcal P$ of positive probability densities, then $\mathcal {H}_p$ is the space of centered square integrable random variables $L^2_0(p)$ and moreover there is an identification of the fiber with the tangent space $\mathcal {H}_p \simeq T_p\mathcal P$. A similar situation occurs even when $\mathcal M$ is a finite-dimensional exponential family. It is difficult to devise set-ups other than those mentioned above, where the identification of the Hilbert fiber with the tangent space holds true. In fact, a necessary condition would be the topological linear isomorphism among fibers. One possible option would be to take as fibers the spaces of bounded functions $L^\infty_0(p)$, see G.~Loaiza and H.R.~Quiceno \cite{loaiza|quiceno:2013-JMAA}. 

This difficulty is overcome in the  N.J.~Newton's setting. On a probability space $(X,\mathcal X, \mu)$, he considers the ``balanced chart'' $\mathcal M \ni p \mapsto \log p + p - 1 \in L^2_0(\mu)$. In this chart, all the tangent spaces are identified with the fixed Hilbert space $L^2_0(\mu)$ so that the statistical Hilbert bundle is trivialized.

N.J.~Newton balanced chart falls in a larger class of ``deformation'' of the usual logarithmic representation. It is in fact an instance of the class of ``deformed logarithm'' as defined by J.~Naudts \cite{naudts:2011GTh}. It is defined as $\log_A(x) = \int_1^x dt/A(t)$, where $A$ is a suitable increasing function. If $A$ is bounded, then a special class of deformed logarithms results. It includes N.J.~Newton balanced chart as well as other deformed logarithms, notably the G.~Kaniadakis logarithm \cite{kaniadakis:2002PhRE,kaniadakis:2005PhRE,pistone:2009EPJB}.

In this paper, we try a mixture of the various approaches by considering deformed logarithms with linear growth as established by N.J.~Newton, but we do not look for a trivialization of the Hilbert bundle. Instead we construct an affine atlas of charts, each one centered at a $p \in \mathcal M$. This is obtained by adapting the construction of the exponential manifold of \cite{pistone:2013GSI} to the deformed exponential models as defined by J.~Naudts \cite{naudts:2011GTh}. Moreover, we allow for a form of general reference measure by using an idea introduced by R.F.~Vigelis and C.C.~Cavalcante \cite{vigelis|cavalcante:2013}. That is, each density has the form $q = \exp_A(u - K_p(u) + \log_A p)$, where $\exp_A = \log_A^{-1}$ is an exponential-like function which has a linear growth at $+\infty$ and is dominated by an exponential at $-\infty$. 

The formalism of deformed exponentials is discussed in Sec.~\ref{sec:deformed}. This section is intended to be self-contained and contains material from the references discussed above without an explicit mention. The following Sec.~\ref{sec:nigel-newt-deform} is devoted to the study of non-parametric deformed exponential families. In Sec.~\ref{sec:convex-conjugate} we introduce the formulation of the divergence, in accordance with our approach. In Sec.~\ref{sec:riem-manif-based} the construction of the Hilbert statistical bundle is outlined.

A first version of this piece of research has been presented at the GSI 2017 Conference \cite{montrucchio|pistone:2017} and we refer to that paper for some of the proofs.

\section{Deformed exponential}\label{sec:deformed}

Let us introduce a class of the deformed exponential, according to the formalism introduced by \cite{naudts:2011GTh}.
Assume to be given a function $A$ from $]0,+\infty[ $ onto $]0,a[$,
strictly increasing, continuously differentiable and such that $\left\Vert A^{\prime }\right\Vert_{\infty } < \infty$. This implies $a = \left\Vert A\right\Vert_{\infty}$ and $A(x) \le
\left\Vert A^{\prime }\right\Vert_{\infty } x$, so that $\int_0^1 d\xi/A(\xi) \  = +\infty$. 

The $A$-logarithm is the function 
\begin{equation*}
\log _{A}(x)=\int_{1}^{x}\frac{d\xi}{A(\xi )}\ ,\quad x\in ]0,+\infty
\lbrack \ .
\end{equation*}
The $A$-logarithm is strictly increasing from $-\infty $ to $+\infty $, its
derivative $\log _{A}^{\prime }(x)=1/A(x)$ is positive and strictly
decreasing for all $x>0$, hence $\log _{A}$ is strictly concave.

By inverting the $A$-logarithm, one obtains the $A$-exponential, $\exp
_{A}=\log _{A}^{-1}$. The function $\exp _{A}\colon ]-\infty ,+\infty
\lbrack \rightarrow ]0,+\infty \lbrack $ is strictly increasing, strictly
convex, and is the solution to the Cauchy problem 
\begin{equation}\label{Aexp}
\exp _{A}^{\prime }(y)=A(\exp _{A}(y)),\quad \exp _{A}(0)=1\ .
\end{equation}
As a consequence, we have the linear bound 
\begin{equation}\label{eq:lip}
\left\vert \exp _{A}(y_{1})-\exp _{A}(y_{2})\right\vert \leq \left\Vert
A\right\Vert _{\infty }\left\vert y_{1}-y_{2}\right\vert \ .
\end{equation}

The behavior of the $A$-logarithm is linear for large arguments and super-logarithmic for small arguments. To derive explicit bounds, set
\begin{equation*}
\alpha_1 = \min_{x\le 1} \frac{A(x)}x \ , \quad \alpha_2 = \max_{x \le 1} \frac{A(x)}x \ ,
\end{equation*}
namely, they are the best constants such that $\alpha_1 x \le A(x) \le \alpha_2 x$ for  $0 < x \le 1$. Note that $\alpha_1 \geq 0$ while $\alpha_2 > 0$. If in addition also $\alpha_1 > 0$, then
\begin{equation}\label{eq:bound1}
  \frac1{\alpha_2} \log x \le \log_A x \le  \frac1{\alpha_1} \log x \ , \quad 0 < x \le 1 \ .
\end{equation}
If otherwise $\alpha_1=0$, the left inequality is true only.

For $x \ge 1$ we have $A(1) \leq A(x) < \left\Vert A\right\Vert_{\infty}$, hence
\begin{equation}\label{eq:bound2}
 \frac1{\left\Vert A\right\Vert_{\infty}}(x-1) < \log_A x \leq \frac1{A(1)}(x-1) \ , \quad x \ge 1 \ .
\end{equation}
Under the assumptions made on the function $A$, the coefficient $\alpha_1 > 0$, if and only if $A'(0+) > 0$. 
\subsection{Examples}
\label{sec:examples}
The main example of $A$-logarithm is the N.J. Newton $A$-logarithm 
\cite{newton:2012}, with 
\begin{equation*}
A(\xi)=1-\frac1{1+\xi}=\frac{\xi}{1+\xi} \ ,
\end{equation*}
so that 
\begin{equation*}
\log_A(x) = \log x + x - 1\ .
\end{equation*}

There is a simple algebraic expression for the product,
\begin{equation*}
  \log_A(x_1x_2) = \log_A(x_1) + \log_A(x_2) + (x_1-1)(x_2-1) \ .
\end{equation*}

Other similar examples are available in the literature. One is a special
case of the G. Kaniadakis' exponential of \cite{kaniadakis:2001PhA}, generated by
\begin{equation*}
A(\xi) = \frac{2\xi^2}{1+\xi^2} \ .
\end{equation*}
It turns out
\begin{equation*}
\log_A x = \frac{x-x^{-1}}2 \ ,
\end{equation*}
whose inverse provides
\begin{equation*}
\exp_A(y) = y + \sqrt{1+y^2} \ .
\end{equation*}

A remarkable feature of the G. Kaniadakis' exponential is the relation
\begin{equation*}
\exp_A(y)\exp_A(-y) = \left(y+\sqrt{1+y^2}\right)\left(-y+\sqrt{1+y^2}
\right) = 1
\end{equation*}

Notice that the $A$ function for N.J. Newton exponential is concave, while the $A$ function of G. Kaniadakis exponential is not.

Another example is $A(\xi) = 1 - 2^{-\xi}$, which gives $\log_A(x) = \log_2(1 - 2^{-x})$ and $\exp_A(y) = \log_2(1+2^y)$. 

Notable examples of deformed exponentials that do not fit into our set of
assumptions are Tsallis q-logarithms, see \cite{tsallis:1988}. For instance,  for $q=1/2$,
\begin{equation*}
\log_{1/2}x = 2\left( \sqrt{x}-1\right) = \int_{1}^{x}\frac{d \xi}{\sqrt{\xi}}.
\end{equation*}

In this case, $\log _{1/2}(0+)=-\int_{0}^{1}d\xi /\sqrt{\xi }=-2$, so that
the inverse is not defined for all real numbers. Tsallis logarithms provide models having heavy tails, which is not the case in our setting.

\subsection{Superposition operator}
\label{sec:superposition-operator}

The deformed exponential will be employed to represent positive probability densities
in the type $p(x) = \exp_A[u(x)]$, where $u$ is a random variable on a
probability space $(\mathbb{X}, \mathcal{X},\mu)$. For this reason, we are
interested in the properties of the \emph{superposition operator} 
\begin{equation}\label{eq:superposition}
S_A \colon u \mapsto \exp_A\circ\, u
\end{equation}
defined in some convenient functional setting. About superposition operators, see e.g. \cite[Ch. 1]{ambrosetti|prodi:1993} and \cite[Ch. 3]{appell|zabrejko:1990}.

It is clear from the Lipschitz condition ~\eqref{eq:lip} that $\exp _{A}(u)\leq 1+\left\Vert
A\right\Vert _{\infty }\left\vert u\right\vert $, which in turn implies that
the superposition operator $S_{A}$ maps $L^{\alpha }(\mu )$ into itself for all $\alpha \in [1,+\infty]$ and the mapping is uniformly Lipschitz
with constant $\left\Vert A\right\Vert _{\infty }$. Notice that we are
assuming that $\mu$ is a finite measure.

The superposition operator $S_{A}\colon L^{\alpha }(\mu )\rightarrow L^{\alpha }(\mu )$ is 1-to-1 and its image consists of all positive random variables $f$ such that $\log _{A}f\in L^{\alpha }(\mu )$. The following proposition is intercepts a more general result \cite{newton:2016}. We give a direct proof here for sake of completeness and because our setting includes deformed logarithms other than the case treated there.

\begin{proposition}
\label{prop:BBA}
\begin{enumerate}
\item For all $\alpha \in [1,\infty]$, the superposition operator $S_A$ of
Eq.~\eqref{eq:superposition} is Gateaux-differentiable with derivative 
\begin{equation}  \label{eq:derivative-of-exp}
d S_A(u)[h] = A(\exp_A(u))h \ .
\end{equation}
\item $S_A$ is Fr\'echet-differentiable from $L^{\alpha}(\mu)$
to $L^{\beta}(\mu)$, for all $\alpha > \beta \ge 1$.
\end{enumerate}
\end{proposition}
\begin{proof}
\begin{enumerate}
\item Eq.~\eqref{Aexp} implies that for each couple of random variables 
$u,h\in L^{\alpha }(\mu )$ 
\begin{equation*}  
\lim_{t\rightarrow 0}t^{-1}\left( \exp _{A}(u+th)-\exp _{A}(u)\right)
-A(\exp _{A}(u))h=0
\end{equation*}
holds point-wise. Moreover, if each $\alpha \in \lbrack 1,\infty \lbrack $, by Jensen inequality we
infer that if $t > 0$ then 
\begin{multline*}
\left\vert t^{-1}\left( \exp _{A}(u+th)-\exp _{A}(u)\right) -A(\exp
_{A}(u))h\right\vert ^{\alpha }\leq \\
t^{-1}\left\vert h\right\vert ^{\alpha }\int_{0}^{t}\left\vert A(\exp
_{A}(u+rh))-A(\exp _{A}(u))\right\vert ^{\alpha }\ dr \leq \left( 2\left\Vert
A\right\Vert _{\infty }\right) ^{\alpha }\left\vert h\right\vert ^{\alpha }\
.
\end{multline*}
Now, dominated convergence forces the limit to hold in $L^{\alpha }(\mu )$. If $t < 0$, it sufficies to replace $h$ with $-h$.

Whenever $\alpha =\infty $, we can use the second-order bound 
\begin{multline*}
\left\vert t^{-1}\left( \exp _{A}(u+th)-\exp _{A}(u)\right) -A(\exp
_{A}(u))h\right\vert = \\
\vert t \vert^{-1}h^{2}\left\vert \int_{0}^{t}(t-r)\frac{d}{dr}A(\exp _{A}(u+rh))\
dr\right\vert \leq \frac{t}{2}\left\Vert h\right\Vert _{\infty
}^{2} \normat \infty {A'} \normat \infty A \ .
\end{multline*}
As $\left\Vert A^{\prime }\cdot A\right\Vert _{\infty }<\infty $, the
RHS goes to 0 as $t\rightarrow 0$ uniformly for each $h\in L^{\infty }(\mu )$.
\item Given $u,h\in L^{\alpha }(\mu )$, thanks again to Taylor
formula,  
\begin{multline*}
\int \left\vert \exp _{A}(u+h)-\exp _{A}(u)-A(\exp _{A}(u))h\right\vert
^{\beta }\ d\mu \leq \\
\int \left\vert h\right\vert ^{\beta }\int_{0}^{1}\left\vert A(\exp
_{A}(u+rh))-A(\exp _{A}(u))\right\vert ^{\beta }\ dr\ d\mu \ .
\end{multline*}
By means of H\"{o}lder inequality, with conjugate exponents $\alpha /\beta $ and 
$\alpha /(\alpha -\beta )$, the RHS is bounded by 
\begin{equation*}
\left( \int \left\vert h\right\vert ^{\alpha }\ d\mu \right) ^{\frac{\beta }{\alpha }}
\left( \iint \left\vert A(\exp _{A}(u+rh))-A(\exp _{A}(u))\right\vert ^{\frac{\alpha \beta }{\alpha -\beta }}\ dr\ d\mu\right) ^{\frac{\alpha -\beta 
}{\alpha }}\ .
\end{equation*}
Consequently, 
\begin{multline*}
\left\Vert h\right\Vert _{L^{\alpha }(\mu )}^{-1}\left\Vert \exp
_{A}(u+h)-\exp _{A}(u)-A(\exp _{A}(u))h\right\Vert _{L^{\beta }(\mu )}\leq \\
\left( \iint \left\vert A(\exp _{A}(u+rh))-A(\exp _{A}(u))\right\vert ^{\frac{\alpha \beta }{\alpha -\beta }}\ dr\ d\mu\right) ^{\frac{\alpha -\beta 
}{\alpha \beta }}\ .
\end{multline*}
In order to show that the RHS vanishes as $\left\Vert h\right\Vert
_{L^{\alpha }(\mu )}\rightarrow 0$, observe that for all $\delta >0$ we have 
\begin{equation*}
\left\vert A(\exp _{A}(u+rh))-A(\exp _{A}(u))\right\vert \leq 
\begin{cases}
2\left\Vert A\right\Vert _{\infty } & \text{always,} \\ 
\normat \infty {A'} \normat \infty A \delta & \text{if $\left\vert h\right\vert \leq \delta $,}
\end{cases}
\end{equation*}
so that, decomposing the double integral as $\iint =\iint_{\left\vert h\right\vert \leq \delta }+\iint_{\left\vert
h\right\vert >\delta }$,
we obtain 
\begin{multline*}
\iint \left\vert A(\exp _{A}(u+rh))-A(\exp _{A}(u))\right\vert ^{\gamma }\ dr\
 d\mu \leq \\
\left( 2\left\Vert A\right\Vert _{\infty }\right) ^{\gamma }\mu \left\{
\left\vert h\right\vert >\delta \right\} +\left(\normat \infty {A'} \normat \infty A \delta \right) ^{\gamma }\leq \\
\left( 2\left\Vert A\right\Vert _{\infty }\right) ^{\gamma }\delta ^{-\alpha
}\int \left\vert h\right\vert ^{\alpha }\ d\mu +\left(\normat \infty {A'} \normat \infty A \delta \right) ^{\gamma }\ ,
\end{multline*}
where $\gamma =\alpha \beta /(\alpha -\beta )$ and we have used Cebi\v{c}ev
inequality. Now it is clear that the last bound implies the conclusion for
each $\alpha <\infty $. The case $\alpha =\infty $ follows \emph{a fortiori}.
\qed\end{enumerate}
\end{proof}

\begin{remark}
It is not generally true that the superposition operator $S_A$ be Fr\'echet differentiable for $\alpha = \beta$, cf. \cite[\S 1.2]{ambrosetti|prodi:1993}. We repeat here the well known counter-example.

Assume $\mu$ is a non-atomic probability measure. For each $\lambda \in \mathbb{R}$ and $\delta > 0$ define the simple
function 
\begin{equation*}
h_{\lambda,\delta}(x) = 
\begin{cases}
\lambda & \text{if $\left\vert x \right\vert \le \delta$,} \\ 
0 & \text{otherwise.}
\end{cases}
\end{equation*}
For each $\alpha \in [1,+\infty[$ we have 
\begin{equation*}
\lim_{\delta \to 0} \left\Vert
h_{\lambda,\delta}\right\Vert_{L^{\alpha}(\mu)}= \lim_{\delta\to 0} \left\vert
\lambda \right\vert \mu\left\{\left\vert x \right\vert \le \delta
\right\}^{1/\alpha} = 0 \ .
\end{equation*}
Differentiability at 0 in $L^{\alpha}(\mu)$ would imply for all $\lambda$ 
\begin{multline*}
0 = \lim_{\delta\to0} \frac{\left\Vert \exp_A(h_{\lambda,\delta}) - 1 -
A(1) h_{\lambda,\delta}\right\Vert_{L^{\alpha}(\mu)}}{\left\Vert
h_{\lambda,\delta}\right\Vert_{L^\alpha(\mu)}} = \\
\lim_{\delta\to0} \frac{\left\vert \exp_A(\lambda) - 1
-A(1)\lambda\right\vert \mu\left\{x | \left\vert x \right\vert \le \delta
\right\}^{1/\alpha}}{ \left\vert \lambda \right\vert \mu\left\{x |
\left\vert x \right\vert \le \delta \right\}^{1/\alpha}} = \left\vert \frac{\exp_A(\lambda) - 1}\lambda - A(1)\right\vert \ ,
\end{multline*}
which is a contradiction.
\end{remark}

\begin{remark}
  Theorems about the differentiability of the deformed exponential are important because of computations like $\derivby \theta \exp_A(v(\theta)) = \exp'_A(v(\theta)) \dot v(\theta)$ are essential for the geometrical theory of statistical models. Several variations in the choice of the combination domain space - image space are possible. Also, one could look at a weaker differentiability property than Frech\'et differentiability. Our choice is motivated by the results of the following sections. A large class of cases is discussed in \cite{newton:2016}
 \end{remark}
 
\begin{remark}  
It would also be worth to study the action of the superposition operator on spaces of differentiable functions, for example Gauss-Sobolev spaces of P.~Malliavin \cite{malliavin:1995}. If $\mu$ is the standard Gaussian measure on $\reals^n$, and $u$ is a differentiable function such that $u, \partiald {x_i} u \in L^2(\mu)$, $i=1,\dots,n$, then it follows that $\exp_A(u) \in L^2(\mu)$ as well as $\partiald {x_i} \exp_A(u) \in L^2(\mu)$, since 
\begin{equation*}
\frac{\partial}{\partial x_i} \exp_A(u(x)) = A(\exp_A(u(x)) \frac{\partial}{\partial x_i} u(x) \ .
\end{equation*}
We do not pursue this line of investigation here.
\end{remark}

\section{Deformed exponential family based on $\exp_A$}

\label{sec:nigel-newt-deform} According to  \cite{vigelis|cavalcante:2013,ay|jost|le|schwachhofer:2017IGbook}, let us define
the deformed exponential curve in the space of positive measures on $(
\mathbb{X},\mathcal{X})$ as follows
\begin{equation*}
  t\mapsto \mu _{t}=\exp _{A}(tu+\log
_{A}p)\cdot \mu \ , \quad u\in L^{1}(\mu ) \ . 
\end{equation*}
We have the following inequality:
\begin{equation*}
\exp _{A}(x+y)\leq
\left\Vert A\right\Vert _{\infty }x^{+}+\exp _{A}(y).
\end{equation*}
Actually, it is true for $x\leq 0$, as being $\exp _{A}$ 
increasing. For $x=x^{+}>0$ the inequality follows from Eq.~\eqref{eq:lip}.
As a consequence, each $\mu _{t}$ is a finite measure, $\mu _{t}(\mathbb{X})\leq t\left\Vert
A\right\Vert _{\infty }\int u^{+}\ d\mu +1$, with $\mu _{0}=p\cdot \mu $.
The curve is actually continuous and differentiable in $L^1(\mu)$ because the point-wise derivative of the density $p_{t}=\exp _{A}(tu+\log _{A}(p))$ is $\dot{p}_{t}=A(p_{t})u$ so that $\left\vert \dot{p}_{t}\right\vert \leq \left\Vert
A\right\Vert _{\infty }\left\vert u\right\vert $. In conclusion $\mu _{0}=p \cdot \mu$ and $\dot{\mu}_{0}=A\left( p\right) u \cdot \mu$.

There are two ways to normalize the density $p_t$ to total mass 1, either dividing by a normalizing constant $Z(t)$ to get the statistical
model
$t \mapsto \exp_A(tu + \log_A p)/Z(t)$ or, subtracting a constant $
\psi(t)$ from the argument to get the model $t \mapsto \exp_A(tu - \psi(t) +
\log_A(p))$. Unlike the standard exponential case, where these two methods lead to the
same result, this is not the case for deformed exponentials where $
\exp_A(\alpha+\beta) \neq \exp_A(\alpha)\exp_A(\beta)$. We choose in the
present paper the latter option.

Here we use the ideas of \cite{naudts:2011GTh,vigelis|cavalcante:2013,ay|jost|le|schwachhofer:2017IGbook}
to construct deformed non-parametric exponential families. Recall that we
are given: the probability space $(\mathbb{X},\mathcal{X},\mu )$; the set $\mathcal{P}$ of the positive probability densities and the function $A$
satisfying the conditions set out in Section \ref{sec:deformed}. Throughout
this section, the density $p\in \mathcal{P}$ will be fixed.

The following proposition is taken from \cite{montrucchio|pistone:2017} where a detailed proof is given.

\begin{proposition}\label{prop:Aexp}
\begin{enumerate}
\item The mapping $L^{1}(\mu )\ni u\mapsto \exp _{A}(u+\log _{A}p)\in
L^{1}(\mu )$ has full domain and is $\left\Vert A\right\Vert _{\infty }$
-Lipschitz. Consequently, the mapping 
\begin{equation*}
u\mapsto \int g\exp _{A}(u+\log _{A}p)\ d\mu
\end{equation*}
is $\left\Vert g\right\Vert _{\infty }\cdot \left\Vert A\right\Vert _{\infty
}$-Lipschitz for each bounded function $g$.

\item For each $u\in L^{1}(\mu )$ there exists a unique constant $
K_{p}(u)\in \mathbb{R}$ such that $\exp _{A}(u-K_{p}(u)+\log _{A}p)\cdot \mu 
$ is a probability.

\item $K_{p}(u)=u$ if, and only if, $u$ is constant. In such a
case, 
\begin{equation*}
\exp _{A}(u-K_{p}(u)+\log _{A}p)\cdot \mu =p\cdot \mu \ .
\end{equation*}
Otherwise, $\exp _{A}(u-K_{p}(u)+\log _{A}p)\cdot \mu \neq p\cdot \mu $.

\item\label{item:Aexp4} A density $q$ is of the form $q=\exp _{A}(u-K_p(u)+\log_A p)$, with $
u\in L^{1}(\mu )$ if, and only if, $\log _{A}q - \log_A p \in L^1(\mu)$.

\item If 
\begin{equation*}
\exp _{A}(u-K_{p}(u)+\log _{A}p)=\exp _{A}(v-K_{p}(v)+\log _{A}p)\ ,
\end{equation*}
 with $u,v\in L^{1}(\mu )$, then $u-v$ is constant.

\item The functional $K_{p}\colon L^{1}(\mu )\rightarrow \mathbb{R}$ is
translation invariant. More specifically,
\begin{equation*}
K_{p}(u+c)=K_{p}(u)+cK_{p}(1)
\end{equation*}
holds for all $c\in \mathbb{R}$.

\item $K_{p}:L^{1}(\mu )\rightarrow \mathbb{R}$ is continuous and convex.
\end{enumerate}
\end{proposition}
\subsection{Escort density}
\label{sec:escortdensity}
For each positive density $q\in \overline{\mathcal P}$, its \emph{escort density} is defined as
\begin{equation*}
  \escortof q = \frac{A(q)}{\int A(q)\ d\mu} \ ,
\end{equation*}
see \cite{naudts:2011GTh}. Notice that $0 \le A(q)\le \normat \infty {A}$. In particular, $\widetilde q = \escortof q$ is a bounded positive density. Hence, $\escortof {\overline{\mathcal P}}\subseteq  \overline{\mathcal P}\cap L^{\infty}(\mu)$. Clearly, the inclusion $\escortof{\mathcal P}\subseteq \mathcal P\cap L^{\infty}(\mu)$ is true as well. 

\begin{proposition}\label{prop:XCM}
  \begin{enumerate}
   \item\label{item:XCM1}
 The mapping $\operatorname{escort} \colon \overline{\mathcal P}\rightarrow  \overline{\mathcal P}\cap L^{\infty}(\mu)$ is a.s. injective. 
    \item\label{item:XCM2}
A bounded positive density $\widetilde q$ is an escort density, i.e., $\widetilde q\in\escortof {\overline{\mathcal P}}$ if, and only if,
\begin{equation}\label{eq:rangecondition}
 \lim_{\alpha \uparrow \normat \infty {A}} \int A^{-1}\left(\alpha \frac{\widetilde q}{\normat \infty {\widetilde q}}\right)\ d\mu \ge 1 \ .
\end{equation}
\item \label{item:XCM3}
Condition \eqref{eq:rangecondition} is fulfilled if $\mu\set{\widetilde q = \normat \infty {\widetilde q}} > 0$. In particular, every density taking a finite number of different values, i.e., a simple density, is an escort density.
\item\label{item:XCM4} If $\widetilde q_1 = \escortof {q_1}$ is an escort density, and $q_2$ is a bounded positive density such that
  \begin{equation*}
    \mu\set{\widetilde q_1 > t \normat \infty {\widetilde q_1}} \leq \mu\set{q_2 > t \normat \infty {q_2}}, \quad t > 0 \ , 
  \end{equation*}
then $q_2$ is an escort density as well.
\end{enumerate}
\end{proposition}
\begin{proof}
  \begin{enumerate}
  \item
  Let $\escortof{q_1} = \escortof{q_2}$ for $\mu$-almost all $x$. Say, $\int A\circ q_{1}\ d\mu \geq \int A\circ
q_{2}\ d\mu $. Then $A(q_{2}(x))\leq A(q_{1}(x))$, for $\mu$-almost all $x$. Since $A$ is strictly increasing, it follows $q_{2}(x)\leq q_{1}(x)$ for $\mu$-almost all $x$, which, in turn, implies $q_{1}=q_{2}$ $\mu$-a.s. because both $\mu $-integrals are equal to 1. Thus the escort mapping is a.s. injective.
  \item
    Fix a $\widetilde{q}\in \overline{\mathcal P}\cap L^{\infty}(\mu)$, and define the function
\begin{equation*}
  f(\alpha) = \int A^{-1}\left(\alpha \frac{\widetilde q}{\normat \infty {\widetilde q}}\right)\ d\mu, \quad \alpha \in [0,\left\Vert A\right\Vert_{\infty}[ \ .
\end{equation*}
It is finite, increasing, continuous and $f(0)=0$.
It is clear that the range condition ~\eqref{eq:rangecondition} is necessary because $\widetilde q = \escortof q$ implies $q = A^{-1}\left(\left(\int A(q)\ d\mu\right)\widetilde q\right)$ and, in turn, $1 =  \int A^{-1}\left(\left(\int A(q)\ d\mu\right)\widetilde q\right) \ d\mu$, given that $q$ is a probability density. If we take $\alpha = \int A(q)\ d\mu \ \normat \infty {\widetilde q} \le \normat \infty A$, the range condition is satisfied. Conversely, if the range condition holds, there exists $\alpha \le \normat \infty A$ such that $q = A^{-1}\left(\alpha \frac{\widetilde q}{\normat \infty {\widetilde q}}\right)$ is a positive probability density whose escort is $\widetilde q$. 
\item This is a special case of Item~\ref{item:XCM2}, in that
  \begin{equation*}
    f(\alpha)= \int A^{-1}\left(\alpha \frac{\widetilde q}{\normat \infty {\widetilde q}}\right)\ d\mu  \ge A^{-1}(\alpha)\mu\set{\widetilde q = \normat \infty {\widetilde q}} \ .
  \end{equation*}
  Therefore, $ f(\alpha) \uparrow +\infty$, as $\alpha \uparrow \normat \infty {A}$.  
\item For each bounded positive density $q$ we have
\begin{multline*}
 \int A^{-1}\left(\frac{q}{\normat \infty {q}}\right)\ d\mu = \int_0^{+\infty} \mu\set{\frac{q}{\normat \infty {q}} > A(t)} \ dt = \\ \int_0^{\normat \infty A} \mu\set{\frac{q}{\normat \infty {q}} > s} \frac1{A'\left(A^{-1}(s)\right)} \ ds \ . 
\end{multline*}
Now the necessary condition of Item~\ref{item:XCM3}. follows from Item~\ref{item:XCM1}. and our assumptions.
\qed
\end{enumerate}
\end{proof}

The previous proposition shows that the range of the escort mapping is uniformly dense as it contains all simple densities. Moreover, in the partial order induced by the rearrangement of the normalized density (that is for each $q$ the mapping $t \mapsto \mu\set{\frac q {\normat \infty q} \
> t}$), it contains the full right interval of each element. But the range of the escort mapping is not the full set of bounded positive densities, unless the $\sigma$-algebra $\mathcal X$ is generated by a finite partition. To provide an example, consider on the Lebesgue unit interval the densities $q_{\delta}(x) \propto (1 - x^{1/\delta})$, $\delta > 0$, and $A(x)=x/(1+x)$. The density $q_{\delta}$ turns out to be an escort if, and only if, $\delta \le 1$.

\subsection{Gradient of the normalization operator $K_p$}
Prop.~\ref{prop:Aexp} shows that the functional $K_{p}$ is a global solution
of an equation. We now study its local properties by the
implicit function theorem as well the related subgradients of the convex function 
$K_p$. We refer to \cite[Part I]{ekeland|temam:1999convex2nd} for the
general theory of convex functions in infinite dimension.

For every $u\in L^{1}(\mu )$, let us write 
\begin{equation}
q(u)=\exp _{A}(u-K_{p}(u)+\log _{A}p)
\end{equation}
while $\widetilde{q}(u) = \escortof{q(u)}$ denotes its escort density.

\begin{proposition}
\ \label{prop:subgradient}
\begin{enumerate}
\item \label{item:subgradient1} The functional $K_{p}\colon L^{1}(\mu )\rightarrow \mathbb{R}$ is
Gateaux-differentiable with derivative 
\begin{equation*}
\left. \frac{d}{dt}K_{p}(u+tv)\right\vert _{t=0}=\int v\widetilde{q}(u)\
d\mu \ .
\end{equation*}
It follows that $K_{p}\colon L^{1}(\mu )\rightarrow \mathbb{R}$ is monotone and globally Lipschitz.
\item For every $u,v\in L^{1}(\mu )$, the inequality 
\begin{equation*}
K_{p}(u+v)-K_{p}(u)\geq \int v\widetilde{q}(u)\ d\mu
\end{equation*}
holds, i.e., the density $\widetilde{q}(u)\in L^{\infty }(\mu )$ is the
unique subgradient of $K_{p}$ at $u$.
\end{enumerate}
\end{proposition}
\begin{proof}
\begin{enumerate}
\item Consider the equation 
\begin{equation*}
F(t,\kappa )=\int \exp _{A}(u+tv-\kappa +\log _{A}p)\ d\mu -1=0,\quad t,\kappa
\in \mathbb{R} \ ,
\end{equation*}
so that $\kappa = K_p(u+tv)$.  Derivations under the integral hold by virtue of the bounds 
\begin{multline*}
\left\vert \frac{\partial }{\partial t}\exp _{A}(u+tv-\kappa +\log
_{A}p)\right\vert = \\
\left\vert A(\exp _{A}(u+tv-\kappa +\log _{A}p))v\right\vert \leq \left\Vert
A\right\Vert _{\infty }\left\vert v\right\vert
\end{multline*}
and 
\begin{equation*}
\left\vert \frac{\partial }{\partial \kappa }\exp _{A}(u+tv-\kappa +\log
_{A}p)\right\vert = \\
\left\vert A(\exp _{A}(u+tv-\kappa +\log _{A}p))\right\vert \leq \left\Vert
A\right\Vert _{\infty }\ .
\end{equation*}
Furthermore, the partial derivative with respect to $\kappa$ is never zero. Thanks to the implicit function theorem, there exists the derivative $\left( d\kappa /dt\right) _{t=0}$
which is the desired Gateaux derivative. Since $\widetilde q(u)$ is positive and bounded, $K_p$ is monotone and globally Lipschitz.
\item Thanks to the convexity of $\exp _{A}$ and the derivation formula, we
have 
\begin{equation*}
\exp _{A}(u+v-K_{p}(u+v)+\log _{A}p)\geq q+A(q)(v-(K_{p}(u+v)-K_{p}(v)))\ ,
\end{equation*}
where $q = \exp_A(u - K_p(u) + \log_A p)$.
If we take $\mu $-integral of both sides, 
\begin{equation*}
0\geq \int vA(q)\ d\mu -(K_{p}(u+v)-K_{p}(v))\int A(q)\ d\mu \ .
\end{equation*}
Isolating the increment $K_{p}(u+v)-K_{p}(v)$, the desired inequality
obtains. Therefore, $\widetilde{q}(u)$ is a subgradient of $K_{p}$ at $u$.
From Item~\ref{item:subgradient1}. we deduce that $\widetilde{q}(u)$ is the unique subgradient and further $\widetilde{q}(u)$ is the Gateaux differential of $K_{p}$ at $u$. \qed
\end{enumerate}
\end{proof}

We can also establish Fr\'echet-differentiability of the functional, under more stringent assumptions.

\begin{proposition}
\label{prop:FAZ} Let $\alpha \geq 2.$
\begin{enumerate}
\item The superposition operator 
\begin{equation*}
L^{\alpha }(\mu )\ni v\mapsto \exp _{A}(v+\log _{A}p)\in L^{1}(\mu )
\end{equation*}
is continuously Fr\'{e}chet differentiable with derivative 
\begin{equation*}
d\exp _{A}(v)=(h\mapsto A(\exp _{A}(v+\log _{A}p))h)\in \mathcal{L}
(L^{\alpha }(\mu ),L^{1}(\mu ))\ .
\end{equation*}
\item The functional $K_{p}:L^{\alpha }(\mu )\rightarrow \mathbb{R}$,
implicitly defined by the equation 
\begin{equation*}
\int \exp _{A}(v-K_{p}(v)+\log _{A}p)\ d\mu =1,\quad v\in L^{\alpha }(\mu )
\end{equation*}
is continuously Fr\'{e}chet differentiable with derivative 
\begin{equation*}
dK_{p}(v)=(h\mapsto \int h\widetilde{q}(v)\ d\mu )\ ,
\end{equation*}
where $\widetilde q(u) = \escortof{q(u)}$.
\end{enumerate}
\end{proposition}
\begin{proof}
\begin{enumerate}
\item Setting $\beta =1$ in Prop.~\ref{prop:BBA}, we
get easily the assertion. It remains just to check that the Fr\'{e}chet
derivative is continuous, i.e., that the Fr\'echet derivative is a continuous
map $L^{\alpha }(\mu )\rightarrow \mathcal{L}(L^{\alpha }(\mu ),L^{1}(\mu ))$. If $\Vert {h}\Vert _{L^{\alpha }(\mu )}\leq 1$ and $v,w\in L^{\alpha }(\mu
)$ we have 
\begin{multline*}
\int \left\vert {(A[\exp _{A}(v+\log _{A}p)]-A[\exp _{A}(w+\log _{A}p)])h}
\right\vert \ d\mu \\
\leq \Vert {A[\exp _{A}(v+\log _{A}p)-A[\exp _{A}(w+\log _{A}p)]}\Vert
_{L^{\sigma }(\mu )}\ ,
\end{multline*}
where $\sigma =\alpha /\left( \alpha -1\right) $ is the conjugate exponent of $
\alpha $. On the other hand,
\begin{eqnarray*}
&&\Vert {A[\exp _{A}(v+\log _{A}p)-A[\exp _{A}(w+\log _{A}p)]}\Vert
_{L^{\sigma }(\mu )} \\
&\leq &\left\Vert A^{\prime }\right\Vert _{\infty }\left\Vert A\right\Vert
_{\infty }\left\Vert v-w\right\Vert _{L^{\sigma }(\mu )}
\end{eqnarray*}
and so the map $L^{\alpha }(\mu )\rightarrow \mathcal{L}(L^{\alpha }(\mu
),L^{1}(\mu ))$ is continuous whenever $\alpha \geq \sigma ,$ i.e., $\alpha
\geq 2$.

\item Fr\'echet differentiability of $K_{p}$ is a consequence of the
Implicit Function Theorem in Banach spaces, see \cite{dieudonne:60}, applied
to the $C^{1}$-mapping 
\begin{equation*}
L^{\alpha }(\mu )\times \mathbb{R}\ni (v,\kappa )\mapsto \int \exp
_{A}(v-\kappa +\log _{A}p)\ d\mu \ .
\end{equation*}
The value of the derivative is given by Prop.~\ref{prop:subgradient}. \qed
\end{enumerate}
\end{proof}

\section {Deformed divergence}
\label{sec:convex-conjugate}
In analogy with the standard exponential case, define the $A$-divergence between
probability densities as
\begin{equation*}
D_{A}(q\Vert p)=\int \left( \log _{A}q-\log _{A}p\right) \escortof{q}\text{ }
d\mu \text{, \ for }q,p\in \mathcal{P} \ .
\end{equation*}

Since $\log_A$ is strictly concave with derivative $1/A$, we have
\begin{equation*}
\log _{A}\left( x\right) \leq \log _{A}\left( y\right) +\frac{1}{A\left(
y\right) }\left( x-y\right)
\end{equation*}
for all $x,y>0$ and with equality if, and only if, $x=y.$ Hence 
\begin{equation}\label{eq:HSA}
A\left( y\right) \left( \log _{A}\left( y\right) -\log _{A}\left( x\right)
\right) \geq y-x\ .  
\end{equation}
It follows in particular that $D_{A}(\cdot \Vert \cdot )$ is a well defined, possibly extended valued, function.

Observe further that by Prop.~\ref{prop:Aexp}.\ref{item:Aexp4}, $\log _{A}q-\log _{A}p \in L^{1}\left( \mu \right) $, and so 
$D_{A}(q\Vert p)<\infty $, whenever $q = q(u)$.

The binary relation $D_{A}$ is a faithful divergence in that it satisfies the following Gibbs' inequality.

\begin{proposition}\label{prop:gibbs}
It holds $D_{A}(q\Vert p)\geq 0$ and $D_{A}(q\Vert p)=0$ if and only if $p=q$.
\end{proposition}

\begin{proof}
From inequality \eqref{eq:HSA} it follows
\begin{align*}
D_{A}(q\Vert p) &=\frac{1}{\int A\left( q\right) d\mu }\int \left( \log
_{A}q-\log _{A}p\right) A\left( q\right) \ d\mu \\
&\geq \frac{1}{\int A\left( q\right) d\mu }\int \left( q-p\right) \text{ } \ d\mu =0.
\end{align*}
Moreover, equality holds if and only if $p=q$ $\mu $-a.e. \qed
\end{proof}

There are other alternative definitions that may fully candidate to be a divergence measure. For instance: 
\begin{equation*}\
I_{A}(q\Vert p)=-\int \log_A(p/q) q \ d\mu.
\end{equation*}
or also
\begin{equation*}\
\widetilde{D}_{A}(q\Vert p)=\int A(q/p)\log_A(p/q) p \ d\mu.
\end{equation*}
By means of the concavity of $\log_A$, it is not difficult to check that both satisfy Gibbs' condition of Prop.~\ref{prop:gibbs}, as well as they equal the Kullback-Leibner functional in the non-deformed case. Observe further that the functional $I_{A}(q\Vert p)$ is closely related to Tallis' divergence (see \cite{tsallis:1988} and also \cite {loaiza|quiceno:2013-JMAA}). In fact, if one replaces $\log_A$ with the q-logarithm, one gets just Tallis' q-divergence.

However our formulation for the divergence is motivated by the structure of the deformed exponential representation. As it will be now seen, our definition of divergence is more adapted to the present setting and it turns out be closely related to the normalizing operator. 

In the equation
\begin{equation} \label{eq:expmodel1}
q = \exp _{A}(u-K_p(u)+\log _{A}p),\quad u\in L^{1}(\mu )\ , \ q \in \mathcal P \ ,
\end{equation}
the random variable $u$ is identified up to an additive constant for any fixed density $q$.
There are at least two options for selecting an interesting representative member
in the equivalence class.

One option is to impose the further condition $\int u\widetilde{p}\ d\mu =0$, where $\widetilde p = \escortof p$, the integral being well defined, given that the escort density is bounded. This restriction provides a unique element $u_q$. On the other hand, if we solve Eq.~\eqref{eq:expmodel1} with respect to $u-K(u)$, we get the desired relation: 
\begin{equation}\label{eq:QMP}
K_{p}(u_q)={\Expectation}_{\widetilde{p}}\left[ \log _{A}p-\log _{A}q\right]
=D_{A}(p\Vert q),  
\end{equation}
where $u=u_q$ is uniquely characterized by the two equations: ${\Expectation}_{\widetilde{p}}\left[ u\right] =0$ and $q=\exp _{A}(u-K_p(u)+\log _{A}p)$.

Observe further that Eq.~\eqref{eq:QMP} entails the relation 
\begin{equation*}
K_{p}(u) = D_{A}(p\Vert q(u))\quad \forall u \in L^{1}\left( \mu \right). 
\end{equation*}
The previous choice is that followed in the construction of the non-parametric exponential manifold, see \cite{pistone|sempi:95,pistone|rogantin:99}.

With regard to the non-deformed case, Eq.~\eqref{eq:QMP} yields the Kulback-Leibler divergence with $p$ and $q$ exchanged, with respect to what is considered more natural in Statistical Physics, see for example the comments \cite{landau|lifshits:1980}.

For this purpose, we undertake another choice for the random variable in the equivalence class. More specifically, in Eq.~\eqref{eq:expmodel1} the random variable $u$ will be now centered with respect to $\widetilde{q} = \escortof q$, i.e., ${\Expectation}_{\widetilde{q}}\left[ u\right] =0$. 

To avoid confusion let us rewrite Eq.~\eqref{eq:expmodel1} as follows and where for convenience the function $K_p$ is replaced with $H_p=-K_p$ : \begin{equation}
q=\exp _{A}(v+H_{p}(v)+\log _{A}p), \quad v\in L^{1}(\mu ),\quad {\Expectation}_{
\widetilde{q}}\left[ v\right] =0,  \label{eq:expmodel2}
\end{equation}
so that 
\begin{equation*}
H_{p}(v_q) ={\Expectation}_{\widetilde{q}}\left[ \log _{A}q-\log _{A}p\right]
=D_{A}(q\Vert p),
\end{equation*}
where $v=v_q$ is the solution to the two equations   ${\Expectation}_{\widetilde{q}}\left[ v\right] =0$ and 
$q=\exp _{A}(v+H_{p}(v)+\log _{A}p)$.
There are hence two notable representations of the same probability density $q$: 
\begin{equation*}
q=\exp_A(u - K_p(u) + \log_A p) = \exp_A(v + H_p(v) + \log_A p)
\end{equation*}
which implies $u_q - v_q = K_p(u_q) + H_p(v_q)$. This, in turn, leads to 
\begin{equation*} 
- {\Expectation}_{\widetilde p}\left[v_q \right] = {\Expectation}_{\widetilde q}\left[u_q \right] =
K_p(u_q)+H_p(v_q)=K_p(u_q)-K_p(v_q).
\end{equation*}
This provides the following remarkable relation
\begin{equation}  \label{eq:conjugate}
H_p(v_q)={\Expectation}_{\widetilde q}\left[u_q \right]\ -\ K_p(u_q).
\end{equation}

\subsection{Variational formula}
We now present a variational formula in the spirit of the classical one by
Donsker-Varadhan. Next proposition provides 
the convex
conjugate of $K_{p}$, in the duality $L^{\infty }(\mu )\times L^{1}(\mu )$.

In what follows, the operator $\eta \mapsto \hat{\eta}$ denotes
the inverse of the escort operator, i.e., $\eta = \escortof {\hat \eta}$. In the light of the results established in Sec.~\ref{sec:escortdensity}, this operator maps a dense subset of $\overline{\mathcal P}\cap L^{\infty}(\mu)$ onto $\overline{\mathcal P}$.  

\begin{proposition}
\begin{enumerate}
\item The convex conjugate function of $K_{p}$ :
\begin{equation}\label{eq:FEN}
K_{p}^{\ast }\left( w\right) =\sup_{u\in L^{1}(\mu )}\left(\int wu\ d\mu
-K_{p}\left( u\right)\right), \quad w\in L^{\infty }(\mu )  
\end{equation}
has domain contained into $\overline{\mathcal P}\cap L^{\infty}(\mu)$. More precisely,
\begin{equation*}
\escortof{\mathcal P}\subseteq domK_{p}^{\ast }\subseteq \overline{\mathcal P}\cap L^{\infty}(\mu).
\end{equation*}

\item 
$K_{p}^{\ast }\left( w\right) \geq 0$ for all $w \in L^{\infty}(\mu)$. For any $\eta \in \escortof{\mathcal P}$,  the conjugate $K_p^*(\eta)$ is given by the \emph{Legendre transform}:
  \begin{equation*}
K_p^*(\eta) = \int \eta\ u_{\hat \eta} \ d\mu - K_p(u_{\hat \eta}) \  .    
\end{equation*}
So that \quad  $K_p^*(\eta)  = H_p(v_{\hat \eta}) = D_A(\hat \eta\Vert p)$ ; equivalently:
 \begin{equation*}
K_p^*(\escortof q)  = D_A(q\Vert p) \quad \forall p,q \in L^1 (\mu ) .   
\end{equation*}

\item It holds the inversion formula
	\begin{align*} 
K_{p}\left( u\right) =\max_{\eta \in  \escortof {\mathcal P} }\left(\int \eta u\ d\mu
- D_A(\hat \eta \Vert p) \right) \\ 
= \max_{q\in \mathcal{P}}  \left(\int \escortof q u\ d\mu
- D_A(q\Vert p) \right),\quad \forall u \in L^{1}(\mu ). 
\end{align*}
\end{enumerate}
\end{proposition}
\begin{proof}
\begin{enumerate} 
\item It follows from the fact that $K_{p}$ is monotone and translation invariant. Let us first suppose $w\notin L^{\infty }_+(\mu )$. That means that
  \begin{equation*}
 \int w \chi_C  \ d\mu < 0    
\end{equation*}
is true for some indicator function $\chi_C$. If we consider the cone generated by the function $-\chi_C$, we can write
\begin{equation*}
K_{p}^{\ast }\left( w\right) \geq \sup_{u\in\ cone(-\chi_C)}\left(\int wu\ d\mu -K_{p}\left(
u\right)\right) \geq \sup_{u\in\ cone(-\chi_C)}\int wu\ d\mu = +\infty,
\end{equation*}
since $K_{p}\left( u\right) \leq 0$ when $u\in\ cone(-\chi_C)$. Now consider the case in which $w \geq 0$. If we set $u = \lambda \in \reals$, we have $K_p(\lambda) = \lambda$ and consequently 
\begin{equation}\label{eq:AAA}
K_p^*(w) \geq \sup_{\lambda \in \reals}\left( \lambda \int w\ d\mu
-\lambda \right) \ .
\end{equation}
This $\sup$ is $+\infty$, unless $\int w \ d\mu = 1$.
Hence, $K_{p}^{\ast }\left( w\right) <\infty $ implies $w\in \overline{\mathcal{P}}$. Summarizing, the domain of $K_{p}^{\ast }$ is contained into $\overline{\mathcal P}\cap L^{\infty}(\mu)$, and this proves one of the two claimed inclusions. The other one will be a direct consequence of the next point.

\item Eq.~\eqref{eq:AAA} implies $K_{p}^{\ast }\geq 0$.  By Prop.~\ref{prop:subgradient} the concave and Gateaux differentiable function $u \mapsto \int \eta u \ d\mu - K_p(u)$ has derivative at $u$ given by $\eta - dK_p(u) = \eta - \escortof{q(u)}$, where $q(u) = \exp_A(u - K_p(u) + \log_Ap)$. Under our assumptions, the derivative vanishes at $u=u_{\hat \eta}$ and the $\sup$ in the definition of $K_p^*$ is attained at that point. The maximum value is $K_p^*(\eta) = \int \eta u \ d\mu - K_p(u)$, by setting $u=u_{\hat \eta}$.

The last formula follows straightforward from Eq.~\eqref{eq:conjugate}.

\item For a well-known property of Fenchel-Moreau duality theory, we have:
\begin{align*}
K_{p}\left( u\right) \geq \int wu\ d\mu - K_p^*(w) \quad \forall u \in L^1 (\mu ), \quad \forall w \in L^{\infty }(\mu ) \\
K_{p}\left( u\right) = \int wu\ d\mu - K_p^*(w) \iff w \in \partial K_{p}\left( u\right).   
\end{align*}
Clearly in our case $\partial K_{p}\left( u\right)$ is a singleton and the image of $\partial K_{p}$ is the set 
$\escortof{\mathcal P}$. Therefore
\begin{equation*}
K_{p}\left( u\right) = \max_{w \in  \escortof {\mathcal P} }\left(\int wu\ d\mu - K_p^*(w) \right).
\end{equation*}
By Item 2 the desired inversion formula obtains.
\qed
\end{enumerate}
\end{proof}

\section{Hilbert bundle based on $\exp_A$}
\label{sec:riem-manif-based}

We shall introduce the Hilbert manifold of probability densities as defined in \cite{newton:2012,newton:2016}. A slightly more general set-up will be introduced, than the one used in that references. By means of a general $A$ function, we provide an atlas of charts, and define a linear bundle as an expression of the tangent space.

Let $\mathcal{P}(\mu )$ denote the set of all $\mu $-densities on the
probability space $(\mathbb{X},\mathcal{X},\mu )$ of the kind 
\begin{equation}\label{eq:Pmu}
q=\exp _{A}(u-K_{1}(u)),\quad u\in L^{2}(\mu ),\quad {\Expectation}_{\mu }\left[ u
\right] =0 \ .
\end{equation}
Notice that $1\in \mathcal{P}(\mu )$ because we can 
take $u=0$.

\begin{proposition}
  \begin{enumerate}
  \item $\mathcal{P}(\mu )$ is the set of all densities $q$ such that $\log _{A}q\in L^{2}(\mu )$, in which case $u = \log_A q - \expectat \mu {\log_A q}$.
  \item If in addition $A'(0+)>0$, then $\mathcal{P}(\mu )$ is the set of all densities $q$ such that both $q$ and $\log q$ are in $L^{2}(\mu )$. 
  \item Let $A'(0+) > 0$. On a product space with reference probability measures $\mu_1$ and $\mu_2$, and densities respectively $q_1$ and $q_2$. We have $q_1 \in \mathcal P(\mu_1)$ and $q_2 \in \mathcal P(\mu_2)$ if, and only if, $q_1 \otimes q_2 \in \mathcal P(\mu_1 \otimes \mu_2)$. 
  \end{enumerate}
\end{proposition}

\begin{proof}
  \begin{enumerate}
  \item 
From Eq.~\eqref{eq:Pmu}, it follows $\log_A q = u - K_1(u) \in L^2(\mu)$, provided $u\in L^{2}(\mu )$.  Conversely, let $\log _{A}q\in L^{2}(\mu )$. Eq.~\eqref{eq:Pmu} yields
\begin{equation*}
  u = \log_A q - K_1(u) \quad and\quad K_1(u)=- \log_A q.
\end{equation*}
Therefore $u = \log_A q - \expectat \mu {\log_A q}$  and $u\in L^{2}(\mu )$.

\item Write

\begin{equation*}
\absoluteval {\log_A q}^2 = \absoluteval{\log_A q}^2 (q < 1) + \absoluteval{\log_A q}^2 (q \geq 1) \,
\end{equation*}
and use the bounds of Eq.~\eqref{eq:bound1} and Eq.~\eqref{eq:bound2} to get
\begin{multline*}
  \expectat \mu {\absoluteval {\log_A q}^2} \le \frac1{\alpha_2^2}   \expectat \mu{\absoluteval{\log q}^2 (q < 1)} + 
 \frac1{A(1)^2}\expectat \mu {\absoluteval{q-1}^2 (q \geq 1)} \leq \\ \frac1{\alpha_2^2} \expectat \mu {\absoluteval{\log q}^2} + \frac1{A(1)^2} \expectat \mu {q^2} \ .
\end{multline*}
We deduce that the two conditions $q$ and $\log q$ in $L^{2}(\mu )$ imply $\log _{A}q\in L^{2}(\mu )$.  

Conversely, let $\log _{A}q\in L^{2}(\mu )$. By means of the other two bounds (recall that $\alpha_1 > 0$) we have too
\begin{equation*}
  \expectat \mu {\absoluteval {\log_A q}^2} \ge \frac1{\alpha_1^2} \expectat \mu {\absoluteval{\log q}^2(q<1)} + \frac1{\left\Vert A\right\Vert_{\infty}^2}\expectat \mu {(q-1)^2(q \ge 1)}\ .
\end{equation*}
Consequently, $\expectat \mu {(q-1)^2(q \ge 1)} < + \infty$. This in turn gives $\expectat \mu {(q-1)^2} < + \infty$, and so $q \in L^{2}(\mu )$.

Once again, the previous inequality provides the condition $\expectat \mu {\absoluteval{\log q}^2(q<1)} < +\infty$. On the other hand, $\expectat \mu {\absoluteval{\log q}^2(q \ge 1)} < +\infty$ since $\absoluteval{\log q}^2(q \ge 1) \le (q-1)^2 (q \ge 1)$. Therefore, $\log q \in L^{2}(\mu )$.   
\item We deduce by the previous item that: $q_1 \otimes q_2 \in \mathcal P(\mu_1 \otimes \mu_2)$ if and only if both $q_1 \otimes q_2$ and $\log (q_1 \otimes q_2)$ are in $L^2(\mu_1\otimes\mu_2)$.

The first condition is equivalent to both $q_1 \in L^2(\mu_1)$ and $q_2 \in L^2(\mu_2)$. The second one is equivalent to $\log q_1 + \log q_2 \in L^2(\mu_1\otimes\mu_2)$. On the other hand, we have
\begin{multline} \label{eq:SSS}
  \expectat {\mu_1 \otimes \mu_2}{(\log q_1 + \log q_2)^2} = \\ \expectat {\mu_1}{\log^2 q_1} + \expectat {\mu_2}{\log^2 q_2} + 2 \  {\expectat {\mu_1} {\log q_1}} {\expectat {\mu_2} {\log q_2}}.
\end{multline}
By Eq.~\eqref{eq:SSS},  $q_1 \in \mathcal P(\mu_1)$ and $q_2 \in \mathcal P(\mu_2)$ imply   
$q_1 \otimes q_2 \in \mathcal P(\mu_1 \otimes \mu_2)$.

Conversely, assume $q_1 \otimes q_2 \in \mathcal P(\mu_1 \otimes \mu_2)$. This implies that it holds,

$\expectat {\mu_1 \otimes \mu_2}{(\log q_1 + \log q_2)^2} < + \infty$. 
Since $\expectat {\mu_i} {\log q_i} \le \expectat {\mu_1} {q_i-1} = 0$. We have $\  {\expectat {\mu_1} {\log q_1}} {\expectat {\mu_2} {\log q_2}} \geq 0$. In view of Eq.~\eqref{eq:SSS}, we can infer that $q_1 \in \mathcal P(\mu_1)$ and $q_2 \in \mathcal P(\mu_2)$ \qed  
\end{enumerate}
\end{proof}

We proceed now to define an Hilbert bundle with base $\mathcal{P}(\mu )$.
The notion of Hilbert bundle has been introduced in Information Geometry by 
\cite{amari:87dual}. We are here using an adaptation to the $A$-exponential of
arguments elaborated by \cite{gibilisco|pistone:98,pistone:2013GSI}. Notice that the construction depends in a essential way on the specific conditions we are assuming for the present class of deformed exponential.

At each $q \in \mathcal{P}(\mu )$ the escort density $\widetilde q$ is bounded, so that we can define the fiber given by the Hilbert spaces 
\begin{equation*}
\mathcal {H}_{q}=\left\{ u\in L^{2}(\mu )|{\Expectation}_{\widetilde{q}}\left[ u\right] =0\right\}
\end{equation*}
with scalar product $\left\langle u,v\right\rangle _{q}=\int uv\ d\mu $. The Hilbert bundle is
\begin{equation*}
H\mathcal{P}(\mu )=\left\{ (q,u)|q\in \mathcal{P}(\mu ),u\in \mathcal {H}_{q}\right\} \
.
\end{equation*}
For each $p,q\in \mathcal{P}(\mu )$ the mapping $\mathbb{U}_{p}^{q}u=u-{\Expectation}_{\widetilde{q}}\left[ u\right] $ is a continuous linear mapping from $\mathcal {H}_{p}$
to $\mathcal {H}_{q}$. Moreover, $\mathbb{U}_{q}^{r}\mathbb{U}_{p}^{q}=\mathbb{U}_{p}^{r}$. In particular, $\mathbb{U}_{q}^{p}\mathbb{U}_{p}^{q}$ is the identity on $\mathcal {H}_{p}$ and so $\mathbb{U}_{p}^{q}$ is an isomorphism of $\mathcal {H}_{p}$ onto $\mathcal {H}_{q}$.

In the next proposition an affine atlas of charts is constructed in order to define our Hilbert bundle which is an expression of the tangent
bundle. The velocity of a curve $t \mapsto p(t) \in \mathcal{P}(\mu)$ is
given in the Hilbert bundle by the so called $A$-score that, in our
case, takes the form $A(p(t))^{-1} \dot p(t)$, where $\dot p(t)$ is computed in $L^1(\mu)$.

The following proposition is taken from \cite{montrucchio|pistone:2017} where a detailed proof is presented.

\begin{proposition}
\begin{enumerate}
\item Fix $p\in \mathcal{P}(\mu )$. A positive density $q \in \mathcal{P}(\mu )$ if and only if
\begin{equation*}
q=\exp _{A}(u-K_{p}(u)+\log _{A}p),\text{ with $u\in L^{2}(\mu )$ and ${\Expectation}_{\widetilde{p}}\left[ u\right] =0$.}
\end{equation*}

\item For any fixed $p\in \mathcal{P}(\mu )$ the mapping $s_p \colon \mathcal{P}(\mu )\rightarrow \mathcal {H}_{p}$ defined by  
\begin{equation*}
q\mapsto \log _{A}q-\log _{A}p+D_{A}(p\Vert
q)
\end{equation*}
is injective and surjective, with inverse $e_{p}(u)=\exp
_{A}(u-K_{p}(u)+\log _{A}p)$.

\item The atlas $\left\{ s_{p}|p\in \mathcal{P}(\mu )\right\} $ is
affine with transitions 
\begin{equation*}
s_{q}\circ e_{p}(u)=\mathbb{U}_{p}^{q}u+s_{p}(q)\ .
\end{equation*}

\item The velocity of the differentiable curve $t\mapsto p(t)\in \mathcal{P}(\mu )$ in the chart $s_{p}$ is $ds_{p}(p(t))/dt\in \mathcal {H}_{p}$. Conversely, given any $u\in \mathcal {H}_{p}$, the curve 
\begin{equation*}
p\colon t\mapsto \exp _{A}(tu-K_{p}(tu)+\log _{A}p)
\end{equation*}
satisfies $p(0)=p$ and has velocity $u$ at $t=0$, expressed in the chart $s_{p}$. If the velocity of a curve is $t\mapsto 
\dot{u}(t)$, in a chart $s_{p}$,   then $\mathbb{U}_{p}^{q}
\dot{u}(t)$ is its velocity in the chart $s_{q}$. 

\item If $t\mapsto p(t)\in \mathcal{P}(\mu )$ is differentiable with
respect to the atlas then it is differentiable as a mapping in $L^{1}(\mu )$. It follows that the $A$-score is well-defined and is the expression of the
velocity of the curve $t\mapsto p(t)$ in the moving chart $t\mapsto s_{p(t)}$.
\end{enumerate}
\end{proposition}

We end here our discussion of the geometry of the Hilbert bundle, because our aim is limited to show the applicability of the analytic results obtained in the previous section. A detailed discussion of the relevant geometric objects e.g., the affine covariant derivative, is not attempted here.

\section{Final remarks}
\label{sec:conclusions}
A non-parametric Hilbert manifold based on a deformed exponential representation of positive densities has been firstly introduced by N.J.~Newton \cite{newton:2012,newton:2016}. We have derived regularity properties of the normalizing functional $K_p$ and discussed the relevant Fenchel conjugation. In particular, we have discussed some properties of the escort mapping a form of the divergence that appears to be especially adapted to our set-up. We have taken a path different from that of N.J.~Newton original presentation. We allow for a manifold defined by an atlas containing charts centered at each density in the model. In conclusion, we have discussed explicitly a version of the Hilbert bundle as a family of codimension 1 sub-vector spaces of the basic Hilbert space.

\begin{acknowledgement}
The Authors wish to thank the anonymous referees whose comments have led to a considerable improvement of the paper. L. Montrucchio is Honorary Fellow of the Collegio Carlo Alberto Foundation. G. Pistone is a member of GNAMPA-INdAM and acknowledges the support of de Castro Statistics and Collegio Carlo Alberto.
\end{acknowledgement}


\begin{thebibliography}{10}
\providecommand{\url}[1]{{#1}}
\providecommand{\urlprefix}{URL }
\expandafter\ifx\csname urlstyle\endcsname\relax
  \providecommand{\doi}[1]{DOI~\discretionary{}{}{}#1}\else
  \providecommand{\doi}{DOI~\discretionary{}{}{}\begingroup
  \urlstyle{rm}\Url}\fi

\bibitem{amari:87dual}
Amari, S.: Dual connections on the {H}ilbert bundles of statistical models.
\newblock In: Geometrization of statistical theory (Lancaster, 1987), pp.
  123--151. ULDM Publ. (1987)

\bibitem{amari|nagaoka:2000}
Amari, S., Nagaoka, H.: Methods of information geometry.
\newblock American Mathematical Society (2000).
\newblock Translated from the 1993 Japanese original by Daishi Harada

\bibitem{ambrosetti|prodi:1993}
Ambrosetti, A., Prodi, G.: A primer of nonlinear analysis, \emph{Cambridge
  Studies in Advanced Mathematics}, vol.~34.
\newblock Cambridge University Press (1993)

\bibitem{appell|zabrejko:1990}
Appell, J., Zabrejko, P.P.: Nonlinear superposition operators, \emph{Cambridge
  Tracts in Mathematics}, vol.~95.
\newblock Cambridge University Press (1990).
\newblock \doi{10.1017/CBO9780511897450}.
\newblock \urlprefix\url{http://dx.doi.org/10.1017/CBO9780511897450}

\bibitem{ay|jost|le|schwachhofer:2017IGbook}
Ay, N., Jost, J., L\^e, H.V., Schwachh\"ofer, L.: Information Geometry.
\newblock Springer (2017)

\bibitem{dieudonne:60}
Dieudonn\'e, J.: Foundations of Modern Analysis.
\newblock Academic press (1960)

\bibitem{ekeland|temam:1999convex2nd}
Ekeland, I., T\'emam, R.: Convex analysis and variational problems,
  \emph{Classics in Applied Mathematics}, vol.~28, english edn.
\newblock Society for Industrial and Applied Mathematics (SIAM) (1999).
\newblock \doi{10.1137/1.9781611971088}.
\newblock \urlprefix\url{http://dx.doi.org/10.1137/1.9781611971088}.
\newblock Translated from the French

\bibitem{gibilisco|pistone:98}
Gibilisco, P., Pistone, G.: Connections on non-parametric statistical manifolds
  by {O}rlicz space geometry.
\newblock IDAQP \textbf{1}(2), 325--347 (1998)

\bibitem{kaniadakis:2001PhA}
Kaniadakis, G.: Non-linear kinetics underlying generalized statistics.
\newblock Physica A \textbf{296}(3-4), 405--425 (2001)

\bibitem{kaniadakis:2002PhRE}
Kaniadakis, G.: Statistical mechanics in the context of special relativity.
\newblock Physical Review E \textbf{66}, 056,125 1--17 (2002)

\bibitem{kaniadakis:2005PhRE}
Kaniadakis, G.: Statistical mechanics in the context of special relativity. ii.
\newblock Phys. Rev. E \textbf{72}(3), 036,108 (2005).
\newblock \doi{10.1103/PhysRevE.72.036108}

\bibitem{kass|vos:1997}
Kass, R.E., Vos, P.W.: Geometrical foundations of asymptotic inference.
\newblock Wiley Series in Probability and Statistics: Probability and
  Statistics. John Wiley \& Sons, Inc., New York (1997).
\newblock \doi{10.1002/9781118165980}.
\newblock \urlprefix\url{http://dx.doi.org/10.1002/9781118165980}.
\newblock A Wiley-Interscience Publication

\bibitem{landau|lifshits:1980}
Landau, L.D., Lifshits, E.M.: Course of Theoretical Physics. Statistical
  Physics., vol.~V, 3rd edn.
\newblock Butterworth-Heinemann (1980)

\bibitem{loaiza|quiceno:2013-JMAA}
Loaiza, G., Quiceno, H.R.: A {$q$}-exponential statistical {B}anach manifold.
\newblock J. Math. Anal. Appl. \textbf{398}(2), 466--476 (2013).
\newblock \doi{10.1016/j.jmaa.2012.08.046}.
\newblock \urlprefix\url{https://doi.org/10.1016/j.jmaa.2012.08.046}

\bibitem{malliavin:1995}
Malliavin, P.: Integration and probability, \emph{Graduate Texts in
  Mathematics}, vol. 157.
\newblock Springer-Verlag (1995).
\newblock With the collaboration of Hélène Airault, Leslie Kay and Gérard
  Letac, Edited and translated from the French by Kay, With a foreword by Mark
  Pinsky

\bibitem{montrucchio|pistone:2017}
Montrucchio, L., Pistone, G.: Deformed exponential bundle: the linear growth
  case.
\newblock In: F.~Nielsen, F.~Barbaresco (eds.) {G}eometric {S}cience of
  {I}nformation, no. 10589 in LNCS, pp. 239--246. Springer (2017).
\newblock Third International Conference, GSI 2017, Paris, France, November
  7-9, 2017, Proceedings

\bibitem{naudts:2011GTh}
Naudts, J.: Generalised thermostatistics.
\newblock Springer-Verlag London Ltd. (2011).
\newblock \doi{10.1007/978-0-85729-355-8}.
\newblock \urlprefix\url{http://dx.doi.org/10.1007/978-0-85729-355-8}

\bibitem{newton:2012}
Newton, N.J.: An infinite-dimensional statistical manifold modelled on
  {H}ilbert space.
\newblock J. Funct. Anal. \textbf{263}(6), 1661--1681 (2012).
\newblock \doi{10.1016/j.jfa.2012.06.007}.
\newblock \urlprefix\url{http://dx.doi.org/10.1016/j.jfa.2012.06.007}

\bibitem{newton:2016}
Newton, N.J.: Infinite-dimensional statistical manifolds based on a balanced
  chart.
\newblock Bernoulli \textbf{22}(2), 711--731 (2016).
\newblock \doi{10.3150/14-BEJ673}.
\newblock \urlprefix\url{https://doi.org/10.3150/14-BEJ673}

\bibitem{pistone:2009EPJB}
Pistone, G.: $\kappa$-exponential models from the geometrical viewpoint.
\newblock The European Physical Journal B Condensed Matter Physics
  \textbf{71}(1), 29--37 (2009).
\newblock \doi{10.1140/epjb/e2009-00154-y}.
\newblock \urlprefix\url{http://dx.medra.org/10.1140/epjb/e2009-00154-y}

\bibitem{pistone:2013GSI}
Pistone, G.: Nonparametric information geometry.
\newblock In: F.~Nielsen, F.~Barbaresco (eds.) Geometric science of
  information, \emph{Lecture Notes in Comput. Sci.}, vol. 8085, pp. 5--36.
  Springer, Heidelberg (2013).
\newblock First International Conference, GSI 2013 Paris, France, August 28-30,
  2013 Proceedings

\bibitem{pistone|rogantin:99}
Pistone, G., Rogantin, M.: The exponential statistical manifold: mean
  parameters, orthogonality and space transformations.
\newblock Bernoulli \textbf{5}(4), 721--760 (1999)

\bibitem{pistone|sempi:95}
Pistone, G., Sempi, C.: An infinite-dimensional geometric structure on the
  space of all the probability measures equivalent to a given one.
\newblock Ann. Statist. \textbf{23}(5), 1543--1561 (1995)

\bibitem{shima:2007}
Shima, H.: The geometry of {H}essian structures.
\newblock World Scientific Publishing Co. Pte. Ltd., Hackensack, NJ (2007).
\newblock \doi{10.1142/9789812707536}.
\newblock \urlprefix\url{http://dx.doi.org/10.1142/9789812707536}

\bibitem{tsallis:1988}
Tsallis, C.: Possible generalization of {B}oltzmann-{G}ibbs statistics.
\newblock J. Statist. Phys. \textbf{52}(1-2), 479--487 (1988)

\bibitem{vigelis|cavalcante:2013}
Vigelis, R.F., Cavalcante, C.C.: On $\phi$-families of probability
  distributions.
\newblock Journal of Theoretical Probability \textbf{26}, 870--884 (2013)

\end{thebibliography}

\end{document}